\documentclass[12pt]{amsart}

\title{Periodic points of polynomials over finite fields}
\author{Derek Garton}
\address{Fariborz Maseeh Department of Mathematics and Statistics, Portland State University}
\email{\href{mailto:gartondw@pdx.edu}{gartondw@pdx.edu}}
\date{\today}

\makeatletter
\@namedef{subjclassname@2020}{\textup{2020} Mathematics Subject Classification}
\makeatother



\usepackage{amssymb,url,MnSymbol,colonequals}

\usepackage{hyperref}
\hypersetup{breaklinks=true,
colorlinks=true,
urlcolor=blue,
linkcolor=cyan
}

\usepackage[top=1in,bottom=1in,left=1in,right=1in]{geometry}

\usepackage[capitalise,nameinlink]{cleveref}

\newcommand{\A}{\ensuremath{\mathbb{A}}}
\newcommand{\C}{\ensuremath{\mathbb{C}}}
\newcommand{\Z}{\ensuremath{\mathbb{Z}}}

\renewcommand{\P}{\ensuremath{\mathbb{P}}}

\newcommand{\R}{\ensuremath{\mathbb{R}}}
\newcommand{\F}{\ensuremath{\mathbb{F}}}

\newcommand{\lv}{\ensuremath{\left\vert}}
\newcommand{\rv}{\ensuremath{\right\vert}}
\newcommand{\lV}{\ensuremath{\left\Vert}}
\newcommand{\rV}{\ensuremath{\right\Vert}}
\newcommand{\lp}{\ensuremath{\left(}}
\newcommand{\rp}{\ensuremath{\right)}}
\newcommand{\lb}{\ensuremath{\left\{}}
\newcommand{\rb}{\ensuremath{\right\}}}
\newcommand{\lc}{\ensuremath{\left[}}
\newcommand{\rc}{\ensuremath{\right]}}

\newcommand{\p}{\ensuremath{\mathfrak{p}}}

\newcommand{\q}{\ensuremath{\mathfrak{q}}}
\newcommand{\qover}{\ensuremath{\mathfrak{Q}}}
\newcommand{\vv}{\ensuremath{\mathfrak{v}}}
\newcommand{\ww}{\ensuremath{\mathfrak{w}}}

\DeclareMathOperator{\Gal}{Gal}
\DeclareMathOperator{\Aut}{Aut}
\DeclareMathOperator{\Per}{Per}
\DeclareMathOperator{\Char}{char}

\DeclareMathOperator{\id}{id}
\DeclareMathOperator{\Frac}{Frac}
\DeclareMathOperator{\Spec}{Spec}
\DeclareMathOperator{\Crit}{Crit}
\DeclareMathOperator{\norm}{N}
\DeclareMathOperator{\Fix}{F}
\DeclareMathOperator{\fix}{f}
\DeclareMathOperator{\ar}{arch}
\DeclareMathOperator{\PGL}{PGL}

\theoremstyle{plain}
\newtheorem{theorem}{Theorem}[section]
\newtheorem{lemma}[theorem]{Lemma}
\newtheorem{corollary}[theorem]{Corollary}
\newtheorem{proposition}[theorem]{Proposition}
\newtheorem{fact}[theorem]{Fact}
\newtheorem{porism}[theorem]{Porism}
\newtheorem*{namedthm}{\namedthmname}
\newcounter{namedthm}

\theoremstyle{remark}
\newtheorem{remark}[theorem]{Remark}

\theoremstyle{definition}
\newtheorem{definition}[theorem]{Definition}

\makeatletter
\newenvironment{named}[1]
  {\def\namedthmname{#1}%
   \refstepcounter{namedthm}%
   \namedthm\def\@currentlabel{#1}}
  {\endnamedthm}
\makeatother

\begin{document}

\begin{abstract}
Fix an odd prime $p$.
If $r$ is a positive integer and $f$ a polynomial with coefficients in $\F_{p^r}$, let $P_{p,r}(f)$ be the proportion of $\P^1\lp\F_{p^r}\rp$ that is periodic with respect to $f$.
We show that as $r$ increases, the expected value of $P_{p,r}(f)$, as $f$ ranges over quadratic polynomials, is less than $22/\lp\log{\log{p^r}}\rp$.
This result follows from a uniformity theorem on specializations of dynamical systems of rational functions over residually finite Dedekind domains.
The specialization theorem generalizes previous work by Juul et al.\ that holds for rings of integers of number fields.
Moreover, under stronger hypotheses, we effectivize this uniformity theorem by using the machinery of heights over general global fields; this version of the theorem generalizes previous work of Juul on polynomial dynamical systems over rings of integers of number fields.
From these theorems we derive effective bounds on image sizes and periodic point proportions of families of rational functions over finite fields.
\end{abstract}

\maketitle

\section{Introduction}
\label{intro}

A \emph{\textup{(}discrete\textup{)} dynamical system} is a pair $\lp S,f\rp$ consisting of a set $S$ and a function $f\colon S\to S$.
For notational convenience, for any positive integer $n$, we let $f^n=\overbrace{f\circ\cdots\circ f}^{n\text{ times}}$; furthermore, we set $f^0=\id_S$.
For any $\alpha\in S$, if there is some positive integer $n$ such that $f^n(\alpha)=\alpha$, we say that $\alpha$ is \emph{periodic} (for $f$).
Let $\Per{\lp S,f\rp}=\lb\alpha\in S\mid\alpha\text{ is periodic for }f\rb$.
One of the results of this paper is the following application of \cref{degtwoave}.
\begin{corollary}\label{easytostate}
Suppose $p$ is an odd prime and $r$ is a positive integer.
If $r>6\log{p}$, then
\[
\frac{1}{\lv\lb f\in\F_{p^r}[X]\mid\deg{f}=2\rb\rv}
\cdot\sum_{\substack{f\in\F_{p^r}[X]\\\deg{f}=2}}
{\frac{\lv\Per{\lp\P^1\lp\F_{p^r}\rp,f\rp}\rv}{\lv\P^1\lp\F_{p^r}\rp\rv}}
<\frac{22}{\log{\log{p^r}}}.
\]
\end{corollary}
\noindent This corollary addresses Question~18.2 in \emph{Current trends and open problems in arithmetic dynamics}, ``To what extent do polynomial maps behave like random set maps?''~\cite{trendz}.
Indeed, the corresponding statistic for random set maps on $\P^1\lp\F_{p^r}\rp$ is: if $p$ is prime, then
\[
\frac{1}
{\lv\Aut_\textsf{Set}{\lp\P^1{\lp\F_{p^r}\rp}\rp}\rv}
\cdot\sum_{f\in\Aut_\textsf{Set}{\lp\P^1{\lp\F_{p^r}\rp}\rp}}
{\frac{\lv\Per{\lp\P^1\lp\F_{p^r}\rp,f\rp}\rv}{\lv\P^1\lp\F_{p^r}\rp\rv}}
=O\lp p^{-\frac{r}{2}}\rp
\]
(see~\cite[Theorem~2]{FO}).
That is, \cref{easytostate} shows that for any odd prime $p$: as $r\to\infty$, the expected value of $\lv\Per{\lp\P^1\lp\F_{p^r}\rp,f\rp}\rv\cdot\lv\P^1\lp\F_{p^r}\rp\rv^{-1}$ approaches 0, whether $f$ ranges over the $p^{3r}-p^{2r}$ quadratic polynomials in $\F_{p^r}[X]$ or the $\lp p^r+1\rp^{p^r+1}$ set maps in $\Aut_\textsf{Set}{\lp\P^1{\lp\F_{p^r}\rp}\rp}$.

For numerical data and conjectures on the statistics of the periodic points of dynamical systems of polynomials of a fixed degree over finite fields (and other statistics), see~\cite{Ketal}.
In~\cite{FG}, the authors bound from below the number of periodic points and number of cycles of these dynamical systems, if the degree of the polynomials grows with the size of the finite field of coefficients; moreover, this bound is consistent with the statistics of random maps.
For fixed-degree polynomials, the authors of~\cite{Betal} prove that a ``noncorrelation'' heuristic implies that the statistics of the numbers of cycles in these dynamical systems match those of random set maps, providing a heuristic justification for Question~18.2 of~\cite{trendz}, mentioned above.

We now mention a different approach to studying the randomness of dynamical systems associated to polynomials over finite fields.
Suppose that $R$ is a commutative ring and $f\in R[X]$.
For any $\p\in\Spec{(R)}$, write $[R]_\p$ for $\Frac{\lp R/\p\rp}$ and $[f]_\p$ for the image of $f$ under the $R$-algebra morphism
\[
R[X]\twoheadrightarrow\lp R/\p\rp[X]\hookrightarrow[R]_\p[X].
\]
With this notation, we see that $\Spec{(R)}$ parameterizes a family of dynamical systems: indeed, to $\p\in\Spec{(R)}$ we associate the dynamical system $\lp\P^1\lp[R]_\p\rp,[f]_\p\rp$.
Of course, if $R$ is residually finite, then (for nonzero primes) this family consists of polynomials acting on a finite set, so we may once again ask: how random are its statistics?
As an example of the utility of this approach, we recall that Pollard's famous ``rho'' method of factorization~\cite{Pollard} relies on an aspect of the purported randomness of the family of dynamical systems associated to $R=\Z$ and $f=X^2+1$.
For recent work on this approach to studying randomness, see~\cite{JKMT,BG1,BG2,JuulP}.
To ease notation, for any commutative ring $R$, we write $\mathcal{P}_R$ for the nonzero prime ideals of $R$; moreover, if $R$ is residually finite, then for any $\p\in\mathcal{P}_R$, we write $\norm{(\p)}$ for $\lv[R]_\p\rv$.
We mention in particular a consequence of~\cite[Proposition~6.4]{JKMT}: if $R$ is the ring of integers of a number field and if there exists $\alpha\in R\setminus\lb-2\rb$ with $f=X^2+\alpha$, then
\[
\lim_{\substack{\p\in\mathcal{P}_R\\
\norm{(\p)}\to\infty}}
{\frac{\lv\Per{\lp\P^1\lp[R]_\p\rp,\lc f\rc_\p\rp}\rv}
{\lv\P^1\lp[R]_\p\rp\rv}}
=0.
\]
Moreover, an effective version of this result follows from \cite[Theorem~1.5~(b)]{JuulP}.
In this paper, we generalize this work of \cite{JKMT} and~\cite{JuulP} to address the case where $R$ is a residually finite Dedekind domain.
One benefit of this generalization is it allows us to bring to bear the techniques of \cite{JKMT,JuulP} on the study of the statistics of the dynamical systems associated to the set of polynomials of fixed degree over a fixed finite field.
For example, if we set $R=\F_p[s]$ and $f=X^2+s\in R[X]$, then for any $r\in\Z_{\geq1}$, the family
\[
\lb\lp\P^1\lp[R]_\p\rp,[f]_\p\rp\mid\p\in\mathcal{P}_R\text{ and }\norm{(\p)}=p^r\rb
\]
parameterizes nearly all (Galois orbits of) dynamical systems in the family
\[
\lb\lp\P^1\lp\F_{p^r}\rp,X^2+\alpha\rp\mid\alpha\in\F_{p^r}\rb.
\]

We now describe the methods, results, and organization of this paper.
In \cref{specialization}, we address the issue of preservation of Galois actions under specialization.
Specifically, if we let
\begin{itemize}
\item
$R$ be a Noetherian integral domain,
\item
$A$ be a finitely-generated $R$-algebra that is an integrally closed Noetherian integral domain, and
\item
$L$ be a finite Galois extension of $\Frac{(A)}$,
\end{itemize}
and we write $K$ for $\Frac{(A)}$ and $B$ for the integral closure of $A$ in $L$, then \cref{reduction} provides criteria than ensure the existence of a nonempty open subset $\mathcal{P}_{A,L}$ of $\Spec{(R)}$ such that for all $\p\in\mathcal{P}_{A,L}$, the ideals $\p A$ and $\p B$ are prime, the extension $[B]_{\p B}/[A]_{\p A}$ is Galois, and the $\Gal{\lp L/K\rp}\simeq\Gal{\lp[B]_{\p B}/[A]_{\p A}\rp}$.
\cref{reduction} is a generalization of Proposition~4.1 of~\cite{JKMT}, which holds when $\Char{(R)}=0$.
As one might expect, the proof of \cref{reduction} must address issues of separability.

As our general results hold for rational functions in addition to polynomials, we pause to introduce some notation.
If $k$ is a field and $\phi\in k(X)$, by the \emph{splitting field of $\phi$ over $k$} we mean the splitting field of any polynomial $f\in k[X]$ with the property that there exists $g\in k[X]$ such that $\phi=f/g$ and $\gcd{(f,g)}=1$.
If $R$ is an integral domain with $k=\Frac{(R)}$, then for any rational function $\phi\in k(X)$, prime $\p\in\Spec{(R)}$, and polynomials $f,g\in R[X]$ with $\phi=f/g$ and $[g]_\p\neq0$, we write $[\phi]_\p$ for $[f]_\p/[g]_\p$.
In \cref{imagesizeone}, we apply \cref{reduction} to the study of the periodic points of dynamical systems of rational functions with coefficients in (the fraction field of) a residually finite integral domain.
To do this, we recall an \ref{effCheb}~\cite[Theorem~2.1]{JuulP}, which states that if a rational function over a finite field satisfies certain Galois constraints, then there are effective bounds on the image size of (iterates of) that function.
Using this result along with \cref{reduction}, we prove in \cref{onenn} that for any
\begin{itemize}
\item
residually finite Dedekind domain $R$ with field of fractions $k$,
\item
rational function $\phi\in k(X)$, and
\item
positive integer $n$,
\end{itemize}
certain Galois hypotheses on the splitting field of $\phi^n(X)-t$ over $k(t)$ imply that for all but finitely many  $\p\in\mathcal{P}_R$, the quantity $\lv[\phi]_\p^n\lp[R]_\p\rp\rv$ is effectively bounded in terms of $\norm{(\p)}$.
This result is a generalization of \cite[Proposition~5.3]{JKMT}, which holds in characteristic zero. 
As a consequence, we obtain \cref{finalgeneral}, which states that if these hypotheses are satisfied for all $n\in\Z_{\geq1}$, then
\[
\lim_{\substack{\p\in\mathcal{P}_R\\
\norm{(\p)}\to\infty}}
{\frac{\lv\Per{\lp\P^1\lp[R]_\p\rp,[\phi]_\p\rp}\rv}
{\lv\P^1\lp[R]_\p\rp\rv}}=0;
\]
this is a generalization of \cite[Corollary~5.4]{JKMT}, which holds when $R$ is the ring of integers of a number field.

In \cref{imagesizetwo}, we prove \cref{wedidit}, an effective version of \cref{onenn} and \cref{finalgeneral}, under the hypothesis that the critical points of $\phi\in k(X)$ lie in $\P^1(k)$ and do not collide under iterations of $\phi$.
This hypothesis ensures that for all $n\in\Z_{\geq1}$, the Galois groups of splitting fields of $\phi^n(X)-t$ over $k(t)$ are certain wreath products (see the \ref{jkmtthm}, due to~\cite{Odoni,JKMT}, in \cref{imagesizetwo}).
Our work in \cref{heights}---specifically, \cref{juulheightgeneral}---guarantees that the critical points of specializations (at primes of large enough norm) of $\phi$ retain this property; to prove it, we introduce the machinery of heights on global fields.
A version of these results has appeared in \cite[Section~7]{JuulP}, in the special case where $\phi$ is a polynomial and $k$ is a number field.

Finally, in \cref{applications}, we apply known statistics of wreath products due to~\cite{JuulP} to deduce our main results.
Indeed, \cref{effectiveimagesize} provides an effective bound on periodic points of rational functions over global fields in the case where the Galois group of the function field extension generated by generic preimages is of the form studied in~\cite{JuulP}.
Turning to more specific applications, we also prove the following theorem.
\begin{theorem}\label{dandm}
Suppose that $q$ is a prime power, that $r,m\in\Z_{\geq1}$, that $d\in\Z_{\geq2}$, and that $\alpha\in\F_{q^r}$.
If $\F_q(\alpha)=\F_{q^r}$, $q\equiv1\pmod{d}$, and $r>\max{\lp\lb2md^2,4d\log_q{(d!)}\rb\rp}$, then
\[
\frac{\lv\Per{\lp \P^1\lp\F_{q^r}\rp,X^d+\alpha^m\rp}\rv}
{\lv\P^1\lp\F_{q^r}\rp\rv}
<\frac{4\log{d}}
{(d-1)\lp\log{\lp\log{q^r}-\log{2}\rp}-\log{\max{\lp\lb\log{q^{2m}},\log{(d!)^4}\rb\rp}}\rp}
+\frac{7d}{q^{\frac{r}{2}}}.
\]
\end{theorem}
\noindent Moreover, in \cref{dandmave} we prove a version of \cref{easytostate} that holds for unicritical polynomials of arbitrarily large degree.
\cref{applications} also contains a proof of the following more precise version of \cref{easytostate}.
\begin{theorem}\label{degtwoave}
Suppose $q$ is a power of an odd prime and $r\in\Z_{\geq1}$.
If $r>8$, then
\begin{align*}
\frac{1}{\lv\lb f\in\F_{q^r}[X]\mid\deg{f}=2\rb\rv}
\cdot\sum_{\substack{f\in\F_{q^r}[X]\\\deg{f}=2}}
{\frac{\lv\Per{\lp\P^1\lp\F_{q^r}\rp,f\rp}\rv}{\lv\P^1\lp\F_{q^r}\rp\rv}}\\
&\hspace{-180px}<\frac{q^r+1}{q^r-1}\lp\frac{\log16}{\log{\lp\log{q^r}-\log{2}\rp}-\log{\max{\lp\lb\log{q^2},\log{16}\rb\rp}}}
+\frac{16}{q^{\frac{r}{2}}}\rp.
\end{align*}
\end{theorem}


\section{Stability of Galois groups under specialization}
\label{specialization}

We start this section with some notation.
For any commutative ring $R$, prime $\p\in\mathcal{P}_R$, and element $\alpha\in R$, we write $[\alpha]_\p$ for the image of $\alpha+\p$ under the canonical injection $R/\p\hookrightarrow[R]_\p$.
Next, suppose that $A$ is an integrally closed domain and $L$ is a Galois extension of $\Frac{A}$, and write $B$ for the integral closure of $A$ in $L$.
If $\q$ a prime ideal of $A$ and $\qover$ a prime ideal of $B$ lying over $\q$, let
\[
D_{L,A}\lp\qover\vert\q\rp
\]
be the decomposition group of $\qover$ over $\q$ and
\[
\rho_{\qover\vert\q}\colon D_{L,A}\lp\qover\vert\q\rp\to\Aut_{[A]_\q}{\lp[B]_\qover/[A]_\q\rp}
\]
be the associated surjective homomorphism of groups.
We are now in the position to recall a fact from algebraic number theory.
The first, second, and fourth bullets below are immediate, and for the third bullet see, for example,~\cite[Proposition~VII.2.8]{Lang}.

\begin{fact}\label{prereduction}
Suppose that $A$ is an integrally closed domain with field of fractions $K$, that $L/K$ is a finite Galois extension, and that $\q$ is a prime ideal of $A$.
Write $B$ for the integral closure of $A$ in $L$ and $\mathfrak{I}$ for $\q B$.
Let $f\in A[X]$ be the minimal polynomial of an integral primitive element of the extension $L/K$.
If $\mathfrak{I}$ is prime and $[f]_\q$ is separable, then
\begin{itemize}
\item
$[B]_\mathfrak{I}/[A]_\q$ is a Galois extension,
\item
$D_{L,A}(\mathfrak{I}\vert\q)=\Gal{(L/K)}$,
\item
$\rho_{\mathfrak{I}\vert\q}$ is an isomorphism of groups, and
\item
for any $\alpha\in B$ and $\sigma\in\Gal{(L/K)}$,
\[
\rho_{\mathfrak{I}\vert\q}\lp\sigma\rp\lp[\alpha]_{\mathfrak{I}}\rp
=[\sigma(\alpha)]_{\mathfrak{I}}.
\]
\end{itemize}
\end{fact}


We now recall two definitions from topology, which we will use only in the following remark and in the proof of \cref{reduction}.
If $\mathbf{X}$ is a Noetherian topological space, we say a subset of $\mathbf{X}$ is \emph{locally closed} if it an intersection of an open set and a closed set.
We say that a subset of $\mathbf{X}$ is \emph{constructible} if it is a finite union of locally closed sets.
The following remark is immediate; we will use it in the proof of~\cref{reduction}.

\begin{remark}\label{constructible}
Suppose $R$ is a Noetherian integral domain and $E$ is a constructible subset of $\Spec{(R)}$.
If $\lb0\rb\in E$, then $E$ contains a nonempty open subset of $\Spec{(R)}$.
\end{remark}

\begin{theorem}\label{reduction}
Suppose that $R$ is a Noetherian integral domain and that $A$ is a finitely-generated $R$-algebra that is an integrally closed integral domain.
Write $k,K$ for the fraction fields of $R,A$, respectively, and suppose that $L/K$ is a finite Galois extension.
Let $B$ be the integral closure of $A$ in $L$ and let $f\in A[X]$ be the minimal polynomial of an integral primitive element of the extension $L/K$.
If
\begin{itemize}
\item
$k$ is algebraically closed in $L$ and
\item
$K$ is separable over $k$,
\end{itemize}
then
\[
\lb\p\in\Spec{(R)}
\mid\p B\text{ is prime and }
[R]_\p\text{ is algebraically closed in }[B]_{\p B}\rb
\]
and
\[
\lb\p\in\Spec{(R)}
\mid[f]_{\p A}\text{ is irreducible and
separable}\rb
\]
both contain nonempty open subsets of $\Spec{(R)}$.
\end{theorem}
\begin{proof}
We begin by showing the first statement.
Let $\mathbf{X}=\Spec{\lp B\rp}$ and for all $\p\in\Spec{(R)}$, let $\mathbf{X}_\p=\Spec{\lp B\otimes_R[R]_\p\rp}$.
Then let
\[
E_1=\lb\p\in\Spec{(R)}\mid \mathbf{X}_\p\text{ is a geometrically integral $[R]_\p$-scheme}\rb.
\]
Since $A$ is an is integrally closed Noetherian domain and $f$ is separable, we know that $B$ is a finitely-generated $A$-algebra (see, for example, Proposition I.6 of~\cite{LangN}), hence a finitely-generated $R$-algebra.
Thus, we apply Theorem~9.7.7 of~\cite{EGA} to conclude that $E_1$ is a constructible set.
Now, for any $\p\in E_1$,
\begin{itemize}
\item
$\p B$ is prime, since $B\otimes_R[R]_\p\simeq B/\p B$, and
\item
$[R]_\p$ is algebraically closed in $[B]_{\p B}$, by Proposition~5.51 of~\cite{GW}, for example.
\end{itemize}
Hence, by \cref{constructible}, to show the first statement we need only show that $\mathbf{X}_{\lb0\rb}$ is geometrically integral.

To see that $\mathbf{X}_{\lb0\rb}$ is integral, let $S=R\setminus\lb0\rb$ and note that
\[
B\otimes_R k
=B\otimes_R \lp S^{-1}R\rp
\simeq S^{-1}B
\]
is a subring of the field $L$, so it is in particular an integral domain.
Thus, as the function field of $\mathbf{X}_{\lb0\rb}$ is $L$, the geometric integrality of $\mathbf{X}_{\lb0\rb}$ will follow if both the following conditions hold:
\begin{itemize}
\item
$k$ is algebraically closed in $L$ and
\item
$L$ is separable over $k$
\end{itemize}
(see, for example, Proposition~5.51 of~\cite{GW}).
But these conditions follow immediately from our hypotheses.

The proof of the second statement is similar.
Let $\mathbf{Y}=\Spec{\lp A[X]/ f\rp}$ and for all $\p\in\Spec{(R)}$, let $\mathbf{Y}_\p=\Spec{\lp\lp A[X]/ f\rp\otimes_R[R]_\p\rp}$.
Then let
\[
E_2=\lb\p\in\Spec{(R)}\mid \mathbf{Y}_\p\text{ is a geometrically integral $[R]_\p$-scheme}\rb.
\]
Since $f$ is the minimal polynomial of an integral element of the extension $L/K$, we know $A[X]/ f$ is a finitely-generated $A$-algebra, hence a finitely-generated $R$-algebra.
Thus, we apply Theorem~9.9.7 of~\cite{EGA} a second time to conclude that $E_2$ is a constructible set.
Now, for $\p\in\Spec{(R)}$, we know
\[
\lp A[X]/ f\rp\otimes_R[R]_\p\simeq\lp A/\p A\rp[X]/[f]_{\p A},
\]
so if $\p\in E_2$, then $[f]_{\p A}$ is irreducible and separable.
Once again, by \cref{constructible}, the second statement will now follow from the geometric integrality of $\mathbf{Y}_{\lb0\rb}$.

As above, the ring $\lp A[X]/ f\rp\otimes_Rk$ is isomorphic to a subring of $B$, so it is an integral domain.
And the function field of $\mathbf{Y}_{\lb0\rb}$ is $L$, so the result follows from our hypotheses and \cite[Proposition~5.51]{GW}.
\end{proof}

\section{Image size of specialized rational functions}
\label{imagesizeone}

Before applying \cref{reduction} to dynamical systems, we recall two facts about the dynamics of rational functions.
And before stating these facts, we mention some notation.
For any field $k$, if $\phi\in k(X)$ and $d\in\Z_{\geq0}$, we will write $\deg{\phi}=d$ if there exist $f,g\in k[X]$ such that
\begin{itemize}
\item
$\phi=f/g$,
\item
$\gcd{(f,g)}=1$, and
\item
$d=\max{\lp\lb\deg{f},\deg{g}\rb\rp}$.
\end{itemize}
We extend the definition of ``separable'' by saying that $\phi$ is \emph{separable over $k$} if there exist $f,g\in k[X]$ such that $\phi=f/g$ and $f$ is separable over $k$.
If $R$ is an integral domain with $k=\Frac{(R)}$, then for any $\phi\in k(X)$ and $\p\in\Spec{(R)}$, we say that $\phi$ has \emph{good reduction at $\p$} if $\deg{[\phi]_\p}=\deg{\phi}$.

\begin{fact}\label{plentygood}
Suppose that $R$ is an integral domain with field of fractions $k$ and that $\phi\in k(X)$.
If we write
\[
\mathcal{R}=\lb\p\in\Spec{(R)}\mid
\phi\text{ has good reduction at }\p\rb,
\]
then $\mathcal{R}$ contains a nonempty open subset of $\Spec{(R)}$.
\end{fact}
\begin{proof}
This fact follows from~\cite[Section~2.4]{Sil}: for any $f,g\in R[X]$ such that $\phi=f/g$ and $f,g$ have no common factors in $k[X]$, if we let $\alpha\in R\setminus\lb0\rb$ be the resultant of $f$ and $g$, then any $\p\in\mathcal{P}_R\setminus\mathcal{R}$ must contain $\alpha$.
\end{proof}

\noindent The following fact was noted in the first paragraph of~\cite[Section~3]{JKMT}.

\begin{fact}\label{sepduh}
Suppose that $k$ is a field and $\phi\in k(X)$.
If $\phi^\prime(X)\neq0$, then for all $n\in\Z_{\geq1}$, the rational function $\phi^n(X)-t$ is separable over $k(t)$.
\end{fact}

With these facts in hand, we prove the following corollary of \cref{reduction}.

\begin{corollary}\label{onelevel}
Suppose that $R$ is a Noetherian integral domain with field of fractions $k$, that $\phi\in k(X)$, and that $n\in\Z_{\geq1}$.
Let
\begin{itemize}
\item
$A=R[t]$,
\item
$K=\Frac{A}=k(t)$,
\item
$L$ be the splitting field of $\phi^n(X)-t$ over $K$, and
\item
$B$ be the integral closure of $A$ in $L$.
\end{itemize}
If
\begin{enumerate}
\item\label{sephyp}
$\phi^\prime(X)\neq0$ and
\item\label{galhyp}
$k$ is algebraically closed in $L$,
\end{enumerate}
then
the subset of $\Spec{(R)}$ consisting of those primes $\p$ such that
\begin{itemize}
\item
$\p B$ is prime,
\item
$[B]_{\p B}/[A]_{\p A}$ is Galois,
\item
$\rho_{\p B\vert\p A}$ is an isomorphism, and
\item
$[R]_\p$ is algebraically closed in $[B]_{\p B}$
\end{itemize}
contains a nonempty open subset of $\Spec{(R)}$.
\end{corollary}
\begin{proof}
Thanks to \cref{sepduh}, hypothesis~\hyperref[sephyp]{(\ref*{sephyp})} implies that rational function $\phi^n(X)-t$ is separable over $K$, so we see that $L$ is separable (hence Galois) over $K$.
Let $f\in A[X]$ be a minimal polynomial of a primitive integral element of the extension $L/K$.
Since $K$ is separable over $k$ by construction, hypothesis \hyperref[galhyp]{(\ref*{galhyp})} ensures that the hypotheses of \cref{reduction} are satisfied, so 
\[
\lb\p\in\mathcal{P}_R\mid
\p B\text{ is prime, }
[f]_{\p A}\text{ is separable, and }
[R]_\p\text{ is algebraically closed in }[B]_{\p B}\rb
\]
contains a nonempty open subset of $\Spec{(R)}$.
The result now follows from \cref{prereduction}.
%
\end{proof}

Before stating the next theorem, we introduce some notation for group actions.
If $\rho$ is an action of a finite group $G$ on a set $S$, we will write
\begin{itemize}
\item
$\Fix{(\rho)}\colonequals\lb\sigma\in G\mid\text{there exists }s\in S\text{ such that }\rho(\sigma)(s)=s\rb$ and
\item
$\fix{(\rho)}\colonequals\lv\Fix{(\rho)}\rv/\lv G\rv$.
\end{itemize}
To extend this notation to Galois actions, for any field $K$ and separable $\psi\in K(X)$, let $L$ be a splitting field of $\psi$ over $K$ and $\rho$ be the action of $\Gal{\lp L/K\rp}$ on the roots of $\psi$ in $L$; we will write $\fix_K{(\psi)}$ for $\fix{(\rho)}$, since this quantity depends only on $K$ and $\psi$.

Versions of the following result have appeared many times, see the proof of Proposition~5.3 in~\cite{JKMT}, the proof of Proposition~3.3 in~\cite{JuulFPP}, and~\cite[Theorem~2.1]{JuulP}.
The theorem below is a special case of \cite[Theorem~2.1]{JuulP} and is based on effective version of the Chebotarev Density Theorem; namely, Proposition~6.4.8 of~\cite{FJ}.

\begin{named}{Effective Image Size Theorem}\label{effCheb}
Suppose that $q$ is a prime power, that $\phi\in\F_q(X)$, and that $n\in\Z_{\geq1}$.
Suppose that $\phi^\prime(X)\neq0$ and let
\begin{itemize}
\item
$d=\deg{\phi}$,
\item
$L$ be a splitting field of $\phi^n(X)-t$ over $\F_q(t)$, and
\item
$G=\Gal{\lp L/\F_q(t)\rp}$.
\end{itemize}
If
\begin{itemize}
\item
$L/\F_q(t)$ is tamely ramified, and
\item
$\F_q$ is algebraically closed in $L$,
\end{itemize}
then
\[
\lv\frac{\lv\phi^n\lp\P^1\lp\F_q\rp\rp\rv}{\fix_{\F_q(t)}{\lp\phi^n(X)-t\rp}}-\lv\P^1\lp\F_q\rp\rv\rv
<\frac{7nd\lv G\rv}{q^{1/2}}.
\]
\end{named}
To apply the \ref{effCheb} to the the specializations considered in \cref{reduction}, we must ensure that these specializations are finite fields---that is, we turn our considerations to residually finite Dedekind domains.
%
The following theorem generalizes \cite[Proposition~5.3]{JKMT} both in its effectivity and in that it holds in positive characteristic.

\begin{theorem}\label{onenn}
Suppose that $R$ is a residually finite Dedekind domain with field of fractions $k$, that $\phi\in k(X)$, and that $n\in\Z_{\geq1}$.
Write
\begin{itemize}
\item
$A$ for $R[t]$,
\item
$K$ for $\Frac{A}=k(t)$, and
\item
$d$ for $\deg{\phi}$.
\end{itemize}
Suppose that
\begin{enumerate}
\item\label{gottasep}
$\phi^\prime(X)\neq0$,
\end{enumerate}
so that $\phi^n(X)-t$ is separable over $K$ by \cref{sepduh}, and let
\begin{itemize}
\item
$L$ be the splitting field of $\phi^n(X)-t$ over $K$ and
\item
$G$ be the Galois group of $L/K$.
\end{itemize}
If
\begin{enumerate}
\setcounter{enumi}{1}
\item\label{gottabeclosed}
$k$ is algebraically closed in $L$, and
\item\label{gottatame}
either $\Char{(k)}=0$ or $\gcd{\lp\Char{(k)},\lv G\rv\rp}=1$,
\end{enumerate}
then there exists $N_{R,\phi,n}\in\Z_{\geq1}$ such that for all $\p\in\mathcal{P}_R$ with $\norm{(\p)}>N_{R,\phi,n}$,
\[
\lv\frac{\lv[\phi]_\p^n\lp\P^1\lp[R]_\p\rp\rp\rv}
{\lv\P^1\lp[R]_\p\rp\rv}-\fix_K{\lp\phi^n(X)-t\rp}\rv
<\frac{7nd\lv G\rv}{\norm{(\p)}^{3/2}}\cdot\fix_K{\lp\phi^n(X)-t\rp}.
\]
\end{theorem}
\begin{proof}
Write $B$ for the integral closure of $A$ in $L$.
Thanks to hypotheses~\hyperref[gottasep]{(\ref*{gottasep})} and~\hyperref[gottabeclosed]{(\ref*{gottabeclosed})}, the hypotheses of \cref{onelevel} satisfied.
Thus, since $R$ is a Dedekind domain, \cref{onelevel} and \cref{plentygood} imply that there exists $N_0\in\Z_{\geq1}$ such that for all $\p\in\mathcal{P}_R$ with $\norm{(\p)}>N_0$,
\begin{itemize}
\item
$\p B$ is prime, the extension $[B]_{\p B}/[A]_{\p A}$ is Galois, and $\rho_{\p B\vert\p A}$ is an isomorphism,
\item
$[R]_\p$ is algebraically closed in $\lc L\rc_{\p B}$, and
\item
$\phi$ has good reduction at $\p$.
\end{itemize}
For any such prime $\p$, since $[B]_{\p B}/[A]_{\p A}$ is the Galois splitting field of the irreducible polynomial $[\phi]_{\p A}^n\lp X\rp-[t]_{\p A}$, we know $[\phi]_{\p A}^n\lp X\rp-[t]_{\p A}$ is separable.
As $\p B$ is prime, we know $D_{L,A}\lp\p B\vert\p A\rp=\Gal{(L/K)}$, so the fact that $\rho_{\p B\vert\p A}$ is an isomorphism implies that the action of $\Gal{(L/K)}$ on the roots of $\phi^n(X)-t$ is isomorphic to the action of $\Gal{\lp[B]_{\p B}/[A]_{\p A}\rp}$ on the roots of $[\phi]_{\p A}^n\lp X\rp-[t]_{\p A}$.
In particular, for all such primes $\p$,
\[
\fix_K{\lp\phi^n\lp X\rp-t\rp}=\fix_{[A]_{\p A}}{\lp[\phi]_{\p A}^n\lp X\rp-[t]_{\p A}\rp}.
\]

Now, if $\Char{(k)}=0$, the fact that $R$ is a Dedekind domain ensures only finitely many $\p\in\mathcal{P}_R$ contain any divisor of $\lv G\rv$; thus, we may choose an integer $N\geq N_0$ such that if $\p\in\mathcal{P}_R$ and $\norm{(\p)}>N$, then $\gcd{\lp\Char{\lp[R]_\p\rp},\lv G\rv\rp}=1$.
If, on the other hand, we are in the case $\Char{(k)}>0$, we set $N=N_0$.
In either case, we know that for all $\p\in\mathcal{P}_R$ with $\norm{(\p)}>N$, the extension $\lc B\rc_{\p B}/[A]_{\p A}$ is tamely ramified (by choice of $N$ if $\Char{(k)}=0$ and by hypothesis \hyperref[gottatame]{(\ref*{gottatame})} if $\Char{(k)}>0$).
For all such $\p$, we know $d=\deg{\lp[\phi]_\p\rp}$ since $\phi$ has good reduction at $\p$; the conclusion now follows from the \ref{effCheb}.
\end{proof}

Thanks to \cref{onenn}, we obtain \cref{finalgeneral} below, showing the scarcity of periodic points in specializations of dynamical systems of rational functions over residually finite Dedekind domains.
(The characteristic zero case of \cref{finalgeneral} is Corollary~5.4 of ~\cite{JKMT}.)

\begin{corollary}\label{finalgeneral}
Suppose that $R$ is a residually finite Dedekind domain with field of fractions $k$ and that $\phi\in k(X)$.
Write
\begin{itemize}
\item
$A$ for $R[t]$ and
\item
$K$ for $\Frac{A}=k(t)$.
\end{itemize}
Suppose that
\begin{enumerate}
\item\label{gottasepgen}
$\phi^\prime(X)\neq0$,
\end{enumerate}
and for every $m\in\Z_{\geq1}$, let
\begin{itemize}
\item
$L_m$ be the splitting field of $\phi^m(X)-t$ over $K$ and
\item
$G_m$ be the Galois group of $L_m/K$.
\end{itemize}
If
\begin{enumerate}
\setcounter{enumi}{1}
\item\label{gottabeclosedgen}
for all $m\in\Z_{\geq1}$, the field $k$ is algebraically closed in $L_m$,
\item\label{gottatamegen}
either $\Char{(k)}=0$ or for all $m\in\Z_{\geq1}$, the integers $\Char{(k)}$ and $\lv G_m\rv$ are coprime, and
\item\label{gottanotfix}
$\lim_{m\to\infty}{\fix_K{\lp\phi^m\lp X\rp-t\rp}}=0$,
\end{enumerate}
then
\[
\lim_{\substack{\p\in\mathcal{P}_R\\
\norm{(\p)}\to\infty}}
{\frac{\lv\Per{\lp\P^1\lp[R]_\p\rp,[\phi]_\p\rp}\rv}
{\lv\P^1\lp[R]_\p\rp\rv}}=0.
\]
\end{corollary}
\begin{proof}
Let $\epsilon>0$.
By hypothesis~\hyperref[gottanotfix]{(\ref*{gottanotfix})}, there exists $n\in\Z_{\geq0}$ such that
\[
\fix_K{\lp\phi^n\lp X\rp-t\rp}<\epsilon/2.
\]
Hypotheses~\hyperref[gottasepgen]{(\ref*{gottasepgen})},~\hyperref[gottabeclosedgen]{(\ref*{gottabeclosedgen})}, and~\hyperref[gottatamegen]{(\ref*{gottatamegen})} allow us to apply \cref{onenn} to find an integer $N_0\in\Z_{\geq1}$ with the property that for all $\p\in\mathcal{P}_R$ with $\norm{(\p)}>N_0$,
\[
\frac{\lv[\phi]_\p^n\lp\P^1\lp[R]_\p\rp\rp\rv}
{\lv\P^1\lp[R]_\p\rp\rv}
<\fix_K{\lp\phi^n(X)-t\rp}
+\frac{7nd\lv G_n\rv}{\norm{(\p)}^{3/2}}\cdot\fix_K{\lp\phi^n(X)-t\rp}.
\]
The result follows by setting $N=\max{\lp\lb N_0,\lp\frac{14nd\lv G_n\rv}{\epsilon}\rp^{\frac{2}{3}}\rb\rp}$ and noting
\[
\Per{\lp\P^1\lp[R]_\p\rp,[\phi]_\p\rp}
\subseteq[\phi]_\p^n\lp\P^1\lp[R]_\p\rp\rp.
\]
\end{proof}

\section{Global fields and heights}
\label{heights}

In \cref{imagesizetwo}, we make \cref{reduction} more effective in the special case where the critical points of the specified rational function do not collide.
As the proof of \cref{wedidit} relies on the theory of heights, we begin this section by recalling basic facts on global fields, following~\cite{AW}.

If $k$ is a field, we will write $\mathcal{M}_k$ for the set of places of $k$ and we say $k$ \emph{satisfies the product formula} if for every $\vv\in\mathcal{M}_k$ there exists an absolute value $\lv\,\cdot\,\rv_\vv\in\vv$ such that for every $\alpha\in k\setminus\lb0\rb$,
\begin{itemize}
\item
$\lv\lb\vv\in\mathcal{M}_k\mid1\neq\lv\alpha\rv_\vv\rb\rv<\infty$ and
\item
$\prod_{\vv\in\mathcal{M}_k}{\lv\alpha\rv_\vv}=1$.
\end{itemize}
A field $k$ is a \emph{global field} if
\begin{itemize}
\item
$k$ satisfies the product formula and
\item
for all $\vv\in\mathcal{M}_k$, either
\begin{itemize}
\item
$\vv$ is archimedean or
\item
$\vv$ is nonarchimedean and discrete, with a finite residue field.
\end{itemize}
\end{itemize}

For any global field $k$, we will write $\mathcal{P}_k$ for the nonarchimedean places of $k$.
Additionally, if $\vv\in\mathcal{P}_k$, then we will write
\begin{itemize}
\item
$[k]_\vv$ for the residue field of $k$ at $\vv$,
\item
$\norm{(\vv)}$ for $\lv[k]_\vv\rv$, and
\item
$v_\vv\colon k^\times\to \Z$ for the associated normalized valuation on $k$.
\end{itemize}
For such a place $\vv$, we define the function
\begin{align*}
\lV\,\cdot\,\rV_\vv\colon k&\to\R_{\geq0}\\
\alpha&\mapsto\begin{cases}
\norm{(\vv)}^{-v_\vv(\alpha)}&\text{if }\alpha\neq0\\
0&\text{if }\alpha=0.
\end{cases}
\end{align*}
Moreover, if $\alpha\in k$ and $\lV\alpha\rV_\vv\leq1$, we will write $[\alpha]_\vv$ for the image of $\alpha$ in $[k]_\vv$; if $[\alpha:\beta]\in\P^1(k)$, we will write
\[
[\alpha:\beta]_\vv\colonequals
\begin{cases}
\lc[\alpha/\beta]_\vv:1\rc
&\text{if }\beta\neq0\text{ and }\lV\alpha/\beta\rV_\vv\leq1\\
\lc1:[\beta/\alpha]_\vv\rc
&\text{if }\alpha\neq0\text{ and }\lV\beta/\alpha\rV_\vv\leq1.
\end{cases}
\]

On the other hand, If $\vv$ is an archimedean place of a global field $k$, then we know by by~\cite{O} that the completion of $k$ with respect to $\vv$ is isometric to either $\R$ or $\C$; in the former situation, we say that $\vv$ is a real place and write $\lV\,\cdot\,\rV_\vv\colon k\to\R_{\geq0}$ for the restriction of the Euclidean absolute value on $\R$ to $k$ and in the latter, we say $\vv$ is a complex place and write $\lV\,\cdot\,\rV_\vv\colon k\to\R_{\geq0}$ for the restriction of the square of the Euclidean absolute value on $\C$ to $k$.
The number of archimedean places of a global field is finite (see, for example, \cite[page 473]{AW}); if there are $r$ real places and $s$ imaginary places, we define $\ar{(k)}\colonequals r+2s$.

\begin{fact}[See Theorem~3 of~\cite{AW}]\label{ArtinWhaplesFact}
If $k$ is a global field, then for all $\alpha\in k\setminus\lb0\rb$,
\[
\prod_{\vv\in\mathcal{M}_k}{\lV\alpha\rV_\vv}=1.
\]
\end{fact}

Turning to heights, for any global field $k$, we define a height function $H_k$ on $k$ by setting
\begin{align*}
H_k\colon k&\to\R_{\geq1}\\
\alpha&\mapsto\prod_{\vv\in\mathcal{M}_k}{\max{\lb1,\lV\alpha\rV_\vv\rb}}.
\end{align*}
We extend this definition to projective space by defining
\begin{align*}
H_k\colon\P^1(k)&\to\R_{\geq1}\\
[\alpha:\beta]&\mapsto
\begin{cases}
H_k(\alpha/\beta)
&\text{if }\beta\neq0\\
H_k(\beta/\alpha)&\text{if }\alpha\neq0.
\end{cases}
\end{align*}
\noindent We now introduce the quantities needed to make \cref{wedidit} explicit.

\begin{definition}\label{exdef}
Suppose that $k$ is a global field.
If $f\in k[X]$, say with $f(X)=\alpha_nX^n+\cdots+\alpha_0$ for some $\alpha_n,\ldots,\alpha_0\in k$, then for any $\vv\in\mathcal{M}_k$ we will write
\[
\lV f\rV_\vv
\colonequals\max{\lp\lb\lV a_n\rV_\vv,\ldots,\lV a_0\rV_\vv\rb\rp}.
\]
If $g\in k[X]$ as well, we write
\[
\lV f,g\rV_\vv\colonequals\max{\lp\lb\lV f\rV_\vv,\lV g\rV_\vv\rb\rp}
\]
and
\[
H_k(f,g)\colonequals\prod_{\vv\in\mathcal{M}_k}{\lV f,g\rV_\vv}.
\]
If $g\neq0$, we write
\[
b_{k,f,g}\colonequals\max{\lp\lb2,(\deg{(f/g)}+1)^{\ar{(k)}}H_k(f,g)\rb\rp}.
\]
Furthermore, if $C$ is a finite subset of $\P^1(k)$, we write
\[
c_{k,f,g,C}\colonequals b_{k,f,g}\cdot\max{\lp\lb H_k(\gamma)\mid\gamma\in C\rb\rp}.
\]
Finally, for any $\epsilon\in\R_{>0}$, define the function
\[
n_{k,f,g,C,\epsilon}\colon\lb\vv\in\mathcal{P}_k\mid\norm{(\vv)}>2^{\ar{(k)}}\rb\to\Z
\]
by setting
\[
n_{k,f,g,C,\epsilon}(\vv)\colonequals
\left\lfloor
\frac{\log{(\log{\lp\norm{(\vv)}\rp}-\ar{(k)}\log2)}
-\log{\max{\lp\lb2\log{\lp c_{k,f,g,C}\rp},
\frac{4\log{(d!)}}{\epsilon}\rb\rp}}}
{2\log{\lp\deg{(f/g)}\rp}}
\right\rfloor.
\]
\end{definition}

With these quantities in hand, we now recall two facts on heights that we will need to prove \cref{wedidit}.

\begin{fact}\label{bomb}
If $k$ is a global field and $\alpha,\beta\in k$, then
\[
H_k(\alpha+\beta)\leq2^{\ar{(k)}}H_k(\alpha)H_k(\beta).
\]
\end{fact}

\begin{remark}\label{bombcaveat}
See \cite[Proposition~1.5.15]{BG} for a proof of this fact for absolute heights on number fields (which are a power of the heights we have defined).
The proof of \cref{bomb} is similar, and easier, in the case $\Char{(k)}>0$.
\end{remark}

\begin{named}{Heights and Iterates Theorem}\label{silverman}
Suppose that $k$ is a global field and $\phi\in k(X)$, say with $\phi=f/g$ for $f,g\in k[X]$.
For all $\gamma\in\P^1(k)$,
\[
H_k(\phi(\gamma))\leq b_{k,f,g}H_k(\gamma)^{\deg{(\phi(X))}}.
\]
\end{named}

\begin{remark}\label{silvermancaveat}
See \cite[Theorem~3.11]{Sil}.
As in \cref{bombcaveat}, we mention that Theorem~3.11 of~\cite{Sil} addresses heights on number fields.
As above, the case when $\Char{(k)}>0$ is similar and easier.
\end{remark}

\noindent The following lemma is a version of the \ref{silverman} that holds for iterates of finite subsets of projective space.

\begin{lemma}\label{setsilverman}
Suppose that $k$ is a global field, that $\phi\in k(X)$, and that $C$ is a finite subset of $\P^1(k)$.
Suppose further that $f,g\in k[X]$ and $\phi=f/g$.
Write $d=\deg{\phi}$ and $c=c_{k,f,g,C}$.
If $d\geq2$, then for all $\gamma\in C$, $n\in\Z_{\geq0}$, and $m\in\lb0,1,\ldots,n\rb$,
\[
H_k(\phi^m(\gamma))<c^{d^n}.
\]
\end{lemma}
\begin{proof}
Write $b$ for $b_{k,f,g}$.
Since $b\geq2$, the lemma is certainly true if $m=0$ or $n=0$, so suppose that $m$ and $n$ are positive.
Then for all $\gamma\in C$,
\begin{align*}
H_k\lp\phi^m(\gamma)\rp
&\leq b^{1+d+\cdots+d^{m-1}}H_k(\gamma)^{d^m}
&&(\text{apply the \ref{silverman} $m$ times})\\
&\leq b^{1+d+\cdots+d^{n-1}}H_k(\gamma)^{d^n}
&&(\text{since $b\geq1$})\\
&<\lp bH_k(\gamma)\rp^{d^n}
&&(\text{since $d\geq2$ and $b\geq2$})\\
&\leq c^{d^n}
&&(\text{by definition of }c_{k,f,g,C}).
\end{align*}
\end{proof}

For any global field $k$, rational function $\phi\in k(X)$, finite subset $C\subseteq\P^1(k)$, and nonarchimedean place $\vv\in\mathcal{P}_k$, it is quite possible that size of the image of $C$ in $\P^1\lp[k]_\vv\rp$ is smaller than $\lv C\rv$.
The following theorem gives a bound ensuring that for nonarchimedean places of large enough norm, this possibility does not arise.
It is a generalization of \cite[Lemma~7.2]{JuulP}, which addresses the case where $k$ is a number field, the rational function $\phi$ is a polynomial, and $C\subseteq\A^1(k)$.

\begin{proposition}\label{juulheightgeneral}
Suppose that $k$ is a global field, that $\phi\in k(X)$, that $C$ is a finite subset of $\P^1(k)$, that $\mathfrak{v}\in\mathcal{P}_k$, and that $n\in\Z_{\geq0}$.
Suppose further that $f,g\in k[X]$ and $\phi=f/g$.
Write $d=\deg{\phi}$ and $c=c_{k,f,g,C}$.
If
\begin{itemize}
\item
$d\geq2$,
\item
$\gamma_1,\gamma_2\in C$,
\item
$m_1,m_2\in\lb0,\ldots,n\rb$,
\item
$\phi^{m_1}\lp\gamma_1\rp\neq\phi^{m_2}\lp\gamma_2\rp$, and
\item
$\lc\phi^{m_1}\lp\gamma_1\rp\rc_\vv=\lc\phi^{m_2}\lp\gamma_2\rp\rc_\mathfrak{v}$,
\end{itemize}
then $\norm{(\mathfrak{v})}<2^{\ar{(k)}}c^{2d^n}$.
\end{proposition}
\begin{proof}
We proceed by analyzing two cases.
\begin{itemize}
\item
Suppose that $\lc\phi^{m_1}\lp\gamma_1\rp\rc_\vv=\lc\phi^{m_2}\lp\gamma_2\rp\rc_\mathfrak{v}=[1:0]_\vv$.
Since $\phi^{m_1}\lp\gamma_1\rp\neq\phi^{m_2}\lp\gamma_1\rp$, either $\phi^{m_1}\lp\gamma_1\rp\neq[1:0]$ or $\phi^{m_2}\lp\gamma_2\rp\neq[1:0]$; without loss of generality, suppose the former is true.
Then there exists $\alpha\in k$ such that $\phi^{m_1}\lp\gamma_1\rp=[1:\alpha]$ and $\lV\alpha\rV_\vv<1$.
Thus, we use \cref{ArtinWhaplesFact} to see that
\[
\norm{(\vv)}
\leq\norm{(\vv)}^{v_\vv(\alpha)}
=\lV\alpha\rV_\vv^{-1}
=\prod_{\substack{\ww\in\mathcal{M}_k\\\ww\neq\vv}}{\lV\alpha\rV_\ww}
\leq H_k(\alpha)
\leq2^{\ar{(k)}}H_k\lp\phi^{m_1}\lp\gamma_1\rp\rp H_k\lp\phi^{m_2}\lp\gamma_2\rp\rp.
\]
\item
On the other hand, suppose there are $\alpha,\beta\in k$ such that $\phi^{m_1}\lp\gamma_1\rp=[\alpha:1]$ and $\phi^{m_2}\lp\gamma_2\rp=[\beta:1]$.
Then $\lV\alpha-\beta\rV_\vv<1$, and we use \cref{ArtinWhaplesFact} and \cref{bomb} to see that
\[
\norm{(\vv)}
\leq\norm{(\vv)}^{v_\vv{(\alpha-\beta)}}
=\lV\alpha-\beta\rV_\vv^{-1}
\leq H_k(\alpha-\beta)
\leq 2^{\ar{(k)}}H_k(\alpha)H_k(\beta).
\]
\end{itemize}
The result now follows from \cref{setsilverman}.
\end{proof}

\section{Effective image size of specialized rational functions}
\label{imagesizetwo}

Before proceeding, we recall some basic facts about wreath products.
Suppose that $G,H$ are finite groups and $T$ is a finite set.
If $\tau\colon H\to \Aut_{\textsf{Set}}{(T)}$ is an action of $H$ on $T$, then we use the coordinate permutation action of $H$ on $G^{\lv T\rv}$ to construct the group $G^{\lv T\rv}\rtimes H$; we denote this group by $G\wr_\tau H$.
Now, if $S$ is a another finite set and $\sigma$ is an action of $G$ on $S$, then $G\wr_\tau H$ acts on $S\times T$ via
\[
\lp\lp g_i\rp_{i\in T},h\rp\colon(s,t)\mapsto(\sigma\lp g_t\rp(s),\tau(h)(t));
\]
we denote this action by $\sigma\wr\tau$.

We may also define an iterated wreath product: if $G$ is a finite group and $\sigma$ is an action of $G$ on a finite set $S$,
\begin{itemize}
\item
write $[G]^1$ for $G$ and $[\sigma]^1$ for $\sigma$, and
\item
for any $n\in\Z_{\geq2}$, write $[G]^n$ for $[G]^{n-1}\wr_{\sigma} G$ and $[\sigma]^n$ for $[\sigma]^{n-1}\wr\sigma$.
\end{itemize}
Thus, for all $n\in\Z_{\geq1}$, we see that $[\sigma]^n$ is an action of $[G]^n$ on $S^n$; in this situation, we will write $\Fix_n{\lp\sigma\rp}$ and $\fix_n{(\sigma)}$ for $\Fix{\lp[\sigma]^n\rp}$ and $\fix{\lp[\sigma]^n\rp}$, respectively.
The following remark follows immediately from the definitions.
\begin{fact}\label{wreathsize}
If $\rho$ is an action of a finite group $G$ on a finite set $S$, then for every $n\in\Z_{\geq1}$,
\[
\lv[G]^n\rv=\lv G\rv^{1+\lv S\rv+\cdots+\lv S\rv^{n-1}}.
\]
\end{fact}

It turns out that the Galois groups of the extensions studied in \cref{onenn} and \cref{finalgeneral} are always subgroups of certain wreath products; we record this fact formally in the \ref{jkmtthm}, below.
The \ref{jkmtthm} also includes a criterion ensuring maximality of Galois groups as well as consequences that follow from maximality.
\ref{jkmtthm}~\hyperref[subgroupofwreath]{(\ref*{subgroupofwreath})} follows from~\cite[Lemma~2.5]{JKMT}, \ref{jkmtthm}~\hyperref[wreathcriterion]{(\ref*{wreathcriterion})} follows from Theorem~3.1 and Remark~3.2 of~\cite{JKMT}, and \ref{jkmtthm}~\hyperref[algclosed]{(\ref*{algclosed})} follows from~\cite[Proposition~3.6]{JKMT}.
Finally, \ref{jkmtthm}~\hyperref[fixedsame]{(\ref*{fixedsame})} follows from~\cite[Lemma~4.1]{Odoni}.
To state the theorem, we introduce the following terminology.
If $k$ is a field and $\phi\in k(X)$, let
\[
\Crit{\lp k,\phi\rp}
=\lb\gamma\in\P^1(k)\mid\gamma\text{ is a critical point of $\phi$}\rb.
\]
We will say that the critical points of $\phi$ are \emph{defined over $k$} if for every field extension $\widehat{k}$ of $k$,
\[
\Crit{\lp\widehat{k},\phi\rp}
=\Crit{\lp k,\phi\rp}.
\]
If $C\subseteq\P^1(k)$ and $n\in\Z_{\geq1}$, we say that $C$ is \emph{$\phi$-disjoint for $n$} if for all $\gamma_1,\gamma_2\in C$ and $m_1,m_2\in\lb0,\ldots,n\rb$,
\[
\text{if $\phi^{m_1}\lp\gamma_1\rp=\phi^{m_2}\lp\gamma_2\rp$, then $\gamma_1=\gamma_2$ and $m_1=m_2$.}
\]
If $C$ is $\phi$-disjoint for every positive integer, we will simply say that $C$ is \emph{$\phi$-disjoint}.

\begin{named}{Wreath Product Theorem}\label{jkmtthm}
Suppose that $k$ is a field, that $\phi\in k(X)$, and that $n\in\Z_{\geq2}$.
Write $K=k(t)$, and for every $m\in\Z_{\geq1}$, let $L_m$ be the splitting field of $\phi^m(X)-t$ over $K$.
Suppose that $\phi^\prime(X)\neq0$, so that for all $m\in\Z_{\geq1}$, the field extension $L_m/K$ is Galois by \cref{sepduh}.
Let $G=\Gal{\lp L_1/K\rp}$ and let $\rho$ be the action of $G$ on the roots of $\phi(X)-t$ in $L_1$.
\begin{enumerate}
\item\label{subgroupofwreath}
Then $\Gal{\lp L_n/K\rp}$ is isomorphic to a subgroup of $[G]^n$.
\item\label{wreathcriterion}
Suppose that $C\subseteq\Crit{(k,\phi)}.$
If
\begin{itemize}
\item
$k$ is algebraically closed in $L_1$,
\item
the critical points of $\phi$ are defined over $k$,
\item
$C$ is $\phi$-disjoint for $n$, and
\item
$\lv\Crit{(k,\phi)}\setminus C\rv\leq1$,
\end{itemize}
then $\Gal{\lp L_n/K\rp}\simeq[G]^n$.
\item\label{whatweknow}
If $\Gal{\lp L_n/K\rp}\simeq[G]^n$, then
\begin{enumerate}
\item\label{fixedsame}
$\fix_K{\lp\phi^n(X)-t\rp}=\fix_n{\lp\rho\rp}$ and
\item\label{algclosed}
if $m\in\lb1,\ldots,n-1\rb$ and $m\mid n$, then $k$ is algebraically closed in $L_m$.
\end{enumerate}
\end{enumerate}
\end{named}

\begin{remark}
A priori, the discussion in \cite[Remark~3.2]{JKMT} implies \ref{jkmtthm}~\hyperref[wreathcriterion]{(\ref*{wreathcriterion})} only in the case where $C=\Crit{(k,\phi)}\setminus\lb\infty\rb$.
However, the version we state above follows easily from their argument.
Indeed, if $\gamma\in\P^1(k)$ and $C=\Crit{(k,\phi)}\setminus\lb\gamma\rb$, choose any coordinate change $\mu\in\PGL{(k)}$ with $\mu(\infty)=\gamma$, let $\phi^\mu=\mu^{-1}\circ\phi\circ\mu$, then apply their remark to $\phi^\mu$ and $\Crit{\lp k,\phi^\mu\rp}\setminus\lb\infty\rb$, noting that $L_n$ is a splitting field for $\lp\phi^\mu\rp^n(X)-\mu^{-1}(t)$ over $k\lp\mu^{-1}(t)\rp=k(t)=K$ and
\[
\Crit{\lp k,\phi^\mu\rp}\setminus\lb\infty\rb
=\lb\mu^{-1}(\delta)\mid\delta\in\Crit{(k,\phi)}\rb\setminus\lb\mu^{-1}(\gamma)\rb
=\lb\mu^{-1}(\delta)\mid\delta\in C\rb.
\]
\end{remark}

\begin{remark}
Similarly, the statement of \cite[Proposition~3.6]{JKMT} is only the $m=1$ case of \ref{jkmtthm}~\hyperref[algclosed]{(\ref*{algclosed})}.
To deduce the more general statement, we apply this case to the rational function $\phi^m$.
\end{remark}

Before proving \cref{wedidit}, we pause to introduce terminology connecting our previous work on residually finite Dedekind domains in \cref{imagesizeone} to our current setting of global fields.
Suppose that $k$ is a global field and $R$ is a subring of $k$.
If $R$ is a Dedekind domain with $\Frac{(R)}=k$, then $R$ is necessarily residually finite, so for every $\p\in\mathcal{P}_R$, we may write $\vv_\p$ for the associated (nonarchimedean) place of $k$.
If it is also the case that the function
\begin{align*}
\mathcal{P}_R&\to\mathcal{P}_k\\
\p&\mapsto\vv_\p
\end{align*}
is cofinite, we will say that $R$ is an \emph{Artin-Whaples subring of $k$}.

\begin{fact}[See Theorem~3 of~\cite{AW}]\label{classification}
If $k$ is a global field, then $k$ has an Artin-Whaples subring.
\end{fact}

\noindent Recall that if $\lp S,\phi\rp$ and $\lp S^\prime,\phi^\prime\rp$ are dynamical systems, we say a function $\sigma\colon S\to S^\prime$ is an \emph{isomorphism of dynamical systems} if $\sigma$ is bijective and $\sigma\circ\phi=\phi^\prime\circ\sigma$.
Of course, if $\lp S,\phi\rp$ and $\lp S^\prime,\phi^\prime\rp$ are isomorphic, then $\sigma\lp\Per{\lp S,\phi\rp}\rp=\Per{\lp S^\prime,\phi^\prime\rp}$.
The following remark is immediate.

\begin{remark}\label{isomorphic}
Suppose that $k$ is a global field, that $\phi\in k(X)$, and that $R$ is an Artin-Whaples subring of $k$.
For all $\p\in\Spec{(R)}$, then dynamical systems $\lp\P^1\lp[R]_\p\rp,[\phi]_\p\rp$ and $\lp\P^1\lp[k]_{\vv_\p}\rp,[\phi]_{\vv_\p}\rp$ are isomorphic via the natural isomorphism of fields $[R]_\p\to[k]_{\vv_\p}$.
\end{remark}

We may now use the height results of \cref{heights} to apply \ref{jkmtthm} and \cref{onelevel} to deduce an effective version of \cref{finalgeneral}.
As~\cref{wedidit} holds for rational functions over arbitrary global fields, it is a generalization of~\cite[Theorem~1.5]{JuulP}, which holds for polynomials over number fields.

\begin{theorem}\label{wedidit}
Suppose that $k$ is a global field, that $\phi\in k(X)$, that $C\subseteq\Crit{\lp k,\phi\rp}$, and that $\epsilon>0$.
Suppose further that $f,g\in k[X]$ and $\phi=f/g$.
Write 
\begin{itemize}
\item
$d$ for $\deg{\phi}$, and
\item
$K$ for $k(t)$.
\end{itemize}
Suppose that
\begin{enumerate}
\item\label{sephypx}
$\phi^\prime(X)\neq0$,
\end{enumerate}
and for every $m\in\Z_{\geq1}$, let
\begin{itemize}
\item
$L_m$ be the splitting field of $\phi^m(X)-t$ over $K$, which is Galois by \cref{sepduh}, and
\item
$G_m$ be $\Gal{\lp L_m/K\rp}$.
\end{itemize}
Finally, write $G$ for $G_1$ and let $\rho$ be the action of $G$ on the roots of $\phi(X)-t$ in $L_1$.
If
\begin{enumerate}
\setcounter{enumi}{1}
\item\label{deghypx}
$\deg{\phi}\geq2$,
\item\label{algclosedatstart}
$k$ is algebraically closed in $L_1$,
\item\label{gottabetame}
either $\Char{(k)}=0$ or $\gcd{\lp\Char{(k)},\lv G\rv\rp}=1$,
\item\label{critinkx}
the critical points of $\phi$ are defined over $k$,
\item\label{collidex}
$C$ is $\phi$-disjoint, and
\item\label{Cbig}
$\lv\Crit{(k,\phi)}\setminus C\rv\leq1$,
\end{enumerate}
then there exists $N\in\Z_{\geq1}$ such that for all $\vv\in\mathcal{P}_k$ with $\norm{(\vv)}>N$,
\begin{itemize}
\item $n_{k,f,g,C,\epsilon}(\vv)\geq1$ and
\item
for all $n\in\lb1,\ldots,n_{k,f,g,C,\epsilon}(\vv)\rb$,
\[
\lv\frac{\lv[\phi]_\vv^n\lp[k]_\vv\rp\rv}
{\lv\P^1\lp[k]_\vv\rp\rv}-\fix_{n}{\lp\rho\rp}\rv
<\frac{7d}{\norm{(\vv)}^{\frac{3}{2}-\epsilon}}.
\]
\end{itemize}
In particular, for all such $\vv$ and $n$,
\[
\frac{\lv\Per{\lp\P^1\lp [k]_\vv\rp,[\phi]_\vv\rp}\rv}
{\lv\P^1\lp[k]_\vv\rp\rv}
<\fix_n{\lp\rho\rp}
+\frac{7d}{\norm{(\vv)}^{\frac{3}{2}-\epsilon}}.
\]
\end{theorem}
\begin{proof}
Use \cref{classification} to choose an Artin-Whaples subring of $k$, and call it $R$.
Let $A=R[t]$ and for every $m\in\Z_{\geq1}$, let $B_m$ be the integral closure of $A$ in $L_m$.

Thanks to hypotheses~\hyperref[sephypx]{(\ref*{sephypx})}, \hyperref[algclosedatstart]{(\ref*{algclosedatstart})}, \hyperref[critinkx]{(\ref*{critinkx})}, \hyperref[collidex]{(\ref*{collidex})}, and \hyperref[Cbig]{(\ref*{Cbig})}, we may apply
\ref{jkmtthm}~\hyperref[wreathcriterion]{(\ref*{wreathcriterion})} and~\hyperref[algclosed]{(\ref*{algclosed})} to deduce that for all $m\in\Z_{\geq0}$,
\begin{itemize}
\item
$G_m\simeq[G]^m$ and
\item
$k$ is algebraically closed in $L_m$.
\end{itemize}
We may now apply the definition of Artin-Whaples subring, \cref{plentygood}, and \cref{onelevel} (thanks to hypotheses~\hyperref[sephypx]{(\ref*{sephypx})} and~\hyperref[algclosedatstart]{(\ref*{algclosedatstart})}) to produce an $N_1\in\Z_{\geq1}$ such that
\begin{itemize}
\item
for all $\vv\in\mathcal{P}_k$ with $\norm{(\vv)}>N_1$, there exists $\p\in\mathcal{P}_R$ with $\vv_\p=\vv$ and
\item
for all $\p\in\mathcal{P}_R$ with $\norm{(\p)}>N_1$,
\begin{itemize}
\item
$\phi$ has good reduction at $\p$,
\item
$\p B_1$ is prime,
\item
$\lc B_1\rc_{\p B_1}$ is a Galois (in particular separable) extension of $[A]_{\p A}$,
\item
$\rho_{\p B_1\vert\p A}$ is an isomorphism, and
\item
$[R]_\p$ is algebraically closed in $\lc B_1\rc_{\p B_1}$.
\end{itemize}
\end{itemize}
Now,
\begin{itemize}
\item
choose any $N_2\in\Z_{\geq1}$ such that for all $q\in\Z$ with $q>N_2$,
\[
\log{\log{q}}<q^{\frac{\epsilon}{2}},
\]
\item
let $N_3=2^{\ar{(k)}}c_{k,f,g,C}^{(2d^2)}$,
\item
let $N_4$ be any integer greater than $2^{\ar{(k)}}(d!)^{4d/\epsilon}$, and
\item
\begin{itemize}
\item
if $\Char{(k)}>0$, let $N_5=1$ and
\item
if $\Char{(k)}=0$, use the fact that $R$ is a Dedekind domain to choose any positive integer $N_5$ with the property that for all $\p\in\mathcal{P}_R$ with $\norm{(\p)}> N_5$,
\[
\gcd{\lp\Char{\lp[R]_{\p}\rp},\lv G\rv\rp}=1.
\]
\end{itemize}
\end{itemize}
Finally, set $N=\max{(\lb N_1,N_2,N_3,N_4,N_5\rb)}$.

Choose any $\p\in\mathcal{P}_R$ with $\norm{(\p)}>N$.
For any $\gamma\in\P^1(k)$, we will write $[\gamma]_\p$ for $[\gamma]_{\vv_\p}$.
Similarly, we write $[C]_\p$ for $\lb[\gamma]_\p\mid\gamma\in C\rb$.
Set
\[
n_\p=n_{k,f,g,C,\epsilon}\lp\vv_\p\rp,
\]
which is defined since $\norm{\lp\vv_\p\rp}=\norm{(\p)}>N_3>2^{\ar{(k)}}$.
Note that
\begin{itemize}
\item
by definition of $N_3$ and $N_4$, it follows that $n_\p\geq1$ and
\item
by definition of $n_\p$, we know $\norm{(\p)}=\norm{\lp\vv_\p\rp}\geq2^{\ar{(k)}}c_{k,f,g,C}^{\lp2d^{2n_\p}\rp}$.
\end{itemize}
Thus, by hypotheses~\hyperref[deghypx]{(\ref*{deghypx})} and~\hyperref[collidex]{(\ref*{collidex})}, and \cref{juulheightgeneral}, we see that $[C]_\p$ is $[\phi]_\p$-disjoint for $2n_\p$.
Now fix any $n\in\lb1,\ldots,n_\p\rb$.
As the definition of $N_1$ ensures that $[R]_\p$ is algebraically closed in $\lc B_1\rc_{\p B_1}$ and $\Gal{\lp\lc B_1\rc_{\p B_1}/[A]_{\p A}\rp}\simeq G$, we use hypotheses~\hyperref[critinkx]{(\ref*{critinkx})} and~\hyperref[Cbig]{(\ref*{Cbig})} to note that the critical points of $[\phi]_\p$ are defined over $[R]_\p$ and $\lv\Crit{\lp[R]_\p,[\phi]_p\rp}\setminus[C]_\p\rv\leq1$.
Thus, we may apply \ref{jkmtthm}~\hyperref[wreathcriterion]{(\ref*{wreathcriterion})} and~\hyperref[algclosed]{(\ref*{algclosed})} to deduce that
\begin{itemize}
\item
$\Gal{\lp\lc B_{2n}\rc_{\p B_{2n}}/[A]_{\p A}\rp}\simeq\lc G\rc^{2n}$,
\item
$[R]_\p$ algebraically closed in $\lc B_{n}\rc_{\p B_{n}}$, and
\item
$\Gal{\lp\lc B_n\rc_{\p B_n}/[A]_{\p A}\rp}\simeq\lc G\rc^{n}$.
\end{itemize}
To see that the extension $\lc B_{n}\rc_{\p B_{n}}/[A]_{\p A}$ is tamely ramified, note that
\begin{itemize}
\item
if $\Char{(k)}>0$, then \cref{wreathsize} and hypothesis~\hyperref[gottabetame]{(\ref*{gottabetame})} imply that
\[
\gcd{\lp\Char{\lp[A]_{\p A}\rp},\lv[G]^{n}\rv\rp}
=\gcd{\lp\Char{(k)},\lv G\rv\rp}=1,
\]
and
\item
if $\Char{(k)}=0$, then \cref{wreathsize} and the definition of $N_5$ imply that
\[
\gcd{\lp\Char{\lp[A]_{\p A}\rp},\lv[G]^{n}\rv\rp}
=\gcd{\lp\Char{\lp[R]_\p\rp},\lv G\rv\rp}
=1.
\]
\end{itemize}
Thus, we apply the \ref{effCheb} and \ref{jkmtthm}~\hyperref[fixedsame]{(\ref*{fixedsame})} to deduce that
\[
\lv\frac{\lv[\phi]_\p^{n}\lp[R]_\p\rp\rv}
{\lv\P^1\lp[R]_\p\rp\rv}-\fix_n{\lp\rho\rp}\rv
<\frac{7nd\lv[G]^n\rv}{\norm{(\p)}^{3/2}}\cdot\fix_n{\lp\rho\rp}.
\]
Finally, note that
\begin{align*}
\lv[G]^n\rv
&=\lv G\rv^{1+\cdots+d^{\lp n-1\rp}}
&&\lp\text{by \cref{wreathsize}}\rp\\
&<\lv G\rv^{2d^{\lp n_\p-1\rp}}
&&\lp\text{since $d\geq2$}\rp\\
&\leq\lp d!\rp^{2d^{\lp n_\p-1\rp}}
&&\lp\text{by definition of $G$}\rp\\
&<\lp d!\rp^{\frac{\epsilon\log{\norm{(\p)}}}{2\log{(d!)}}}
&&\lp\text{by definition of $n_\p$}\rp\\
&=\norm{(\p)}^{\frac{\epsilon}{2}}.
\end{align*}
Since $\fix_n{\lp\rho\rp}\leq1$ and $n\leq n_\p\leq\log{\lp\log{\norm{(\p)}}\rp}$, the theorem follows by the definition of $N_2$ and \cref{isomorphic}.

\end{proof}
\noindent We now take a moment to remark that the construction of the constant $N$ in the proof of \cref{wedidit} is almost entirely explicit, especially in positive characteristic; taking advantage of this explicitness leads to the following porism.

\begin{porism}\label{effectivewedidit}
Keep the notation and hypotheses of \cref{wedidit}, and suppose further that $N_1\in\Z_{\geq1}$ and $R$ is an Artin-Whaples subring of $K$ such that for all $\vv\in\mathcal{P}_k$ with $\norm{(\vv)}>N_1$, there exists $\p\in\mathcal{P}_R$ such that $\vv=\vv_\p$.
Write $A$ for $R[t]$ and $B_1$ for the integral closure of $A$ in $L_1$.
Suppose that $\Char{(k)}>0$ and for all $\p\in\mathcal{P}_R$ with $\norm{(\p)}>N_1$,
\begin{itemize}
\item
$\phi$ has good reduction at $\p$,
\item
$\p B_1$ is prime,
\item
$\lc B_1\rc_{\p B_1}$ is a Galois extension of $[A]_{\p A}$,
\item
$\rho_{\p B_1\vert\p A}$ is an isomorphism, and
\item
$[R]_\p$ is algebraically closed in $\lc B_1\rc_{\p B_1}$.
\end{itemize}
Then for all $\vv\in\mathcal{P}_k$ with $\norm{(\vv)}>\max{\lp\lb N_1,c_{k,f,g,C}^{2d^2},(d!)^{4d}\rb\rp}$,
\begin{itemize}
\item $n_{k,f,g,C,1}(\vv)\geq1$ and
\item
for all $n\in\lb1,\ldots,n_{k,f,g,C,1}(\vv)\rb$,
\[
\lv\frac{\lv[\phi]_\vv^n\lp[k]_\vv\rp\rv}
{\lv\P^1\lp[k]_\vv\rp\rv}-\fix_{n}{\lp\rho\rp}\rv
<\frac{7d}{\norm{(\vv)}^{\frac{1}{2}}}.
\]
\end{itemize}
In particular, for all such $\vv$ and $n$,
\[
\frac{\lv\Per{\lp\P^1\lp [k]_\vv\rp[\phi]_\vv\rp}\rv}
{\lv\P^1\lp[k]_\vv\rp\rv}
<\fix_n{\lp\rho\rp}
+\frac{7d}{\norm{(\vv)}^{\frac{1}{2}}}.
\]
\end{porism}
\begin{proof}
This follows immediately from the proof of \cref{wedidit}: set $\epsilon=1$ and compute $N_2,N_3,N_4,N_5$ while keeping in mind that $\ar{(k)=0}$.
\end{proof}

\section{Applications}
\label{applications}

We now apply \cref{wedidit} and \cref{effectivewedidit} to deduce the results mentioned in the introduction, taking advantage of the fact that for many group actions $\rho$ and integers $n\in\Z_{\geq1}$, Juul~\cite{JuulP} has computed effective bounds for $\fix_n{(\rho)}$.

\begin{definition}
For any $d\in\Z_{\geq1}$, let $\rho_{S,d}$, $\rho_{A,d}$, $\rho_{D,d}$, $\rho_{C,d}$ denote the usual actions of $S_d$, $A_d$, $D_d$, $\Z/d\Z$ on $\lb1,\ldots,d\rb$.
\end{definition}
\noindent The following result records the upper bounds proved in Proposition~4.2, Proposition~4.5, Proposition~4.8, and Proposition~4.10 of~\cite{JuulP}.

\begin{named}{Juul's Wreath Bounds}\label{jfixed}
Suppose $d\in\Z_{\geq2}$.
For all $n\in\Z_{\geq1}$,
\begin{enumerate}
\item
\[
\fix_n{\lp\rho_{S,d}\rp}\leq\frac{2}{n+2},
\]
\item
if $d\geq5$, then
\[
\fix_n{\lp\rho_{A,d}\rp}\leq\frac{2}{n+2},
\]
and
\[
\fix_n{\lp\rho_{A,4}\rp}\leq\frac{2}{n+1-\log{n}},
\]
\item
if $d\geq3$, then
\[
\fix_n{\lp\rho_{D,d}\rp}\leq\frac{2}{n+2},
\]
and
\item
\[
\fix_n{\lp\rho_{C,d}\rp}\leq\frac{2}{(d-1)(n+1)}.
\]
\end{enumerate}
\end{named}

\noindent The combination of \cref{wedidit} and \ref{jfixed} immediately implies the following theorem.

\begin{theorem}\label{effectiveimagesize}
Keep the notation and hypotheses of \cref{wedidit}, and assume that $\rho$ is isomorphic to $\rho_{S,d}$, $\rho_{A,d}$, $\rho_{D,d}$, or $\rho_{S,d}$.
Then there exists $N\in\Z_{\geq1}$ such that for all $\vv\in\mathcal{P}_k$ with $\norm{(\vv)}>N$,
\[
\frac{\lv\Per{\lp\P^1\lp[k]_\vv\rp,[\phi]_\vv\rp}\rv}
{\lv\P^1\lp[k]_\vv\rp\rv}
<\begin{cases}
\frac{4\log{d}}{\log{\log{\norm{(\vv)}}}}
+\frac{7d}{\norm{(\vv)}^{\frac{3}{2}-\epsilon}}
&\text{if }\rho\nsimeq\rho_{A,4},\\
\frac{1+4\log{d}}{\log{\log{\norm{(\vv)}}}}
+\frac{7d}{\norm{(\vv)}^{\frac{3}{2}-\epsilon}}
&\text{if }\rho\simeq\rho_{A,4}.
\end{cases}
\]
\end{theorem}
\begin{proof}
This follows immediately from \cref{wedidit}, \ref{jfixed}, and \cref{exdef}.
\end{proof}

To apply~\cref{effectivewedidit}, we now turn to the special case where $q$ is an odd prime and $k=\F_q(s)$.

\begin{proof}[Proof of \cref{dandm}]
Set
\begin{itemize}
\item
$R=\F_q[s]$,
\item
$k=\F_q(s)$,
\item
$A=\F_q[s,t]$, and
\item
$K=\F_q(s,t)$.
\end{itemize}
For any monic irreducible $\pi\in R$, we write $\p_\pi$ for the prime ideal $\pi R$ and $\vv_\pi$ for the associated $\p_\pi$-adic place of $k$ in $\mathcal{P}_k$.
Let $\phi(X)=X^d+s^m$, let $f(X)=\phi(X)-t$, let $L_1$ be the splitting field of $f$ over $K$, let $B_1$ the integral closure of $A$ in $L_1$, let $u$ be any root of $f$ lying in $L_1$, let $G=\Gal{\lp L_1/K\rp}$, and let $\rho$ be the action of $G$ on the roots of $f$ in $L$.
Since $q\equiv1\pmod{d}$, we know $K$ contains a primitive $d$th root of unity, so that $L_1=K(u)=\F_q(s,u)$ and $G\simeq\Z/d\Z$.
Moreover, as $\F_q[s,u]$ is integrally closed in $\F_q(s,u)$, we see that $B_1=\F_q[s,u]$.

Now, for any monic irreducible polynomial $\pi\in R$, we know
\[
B_1/{\p_\pi B_1}
=\F_q[s,u]/\pi(s)\F_q[s,u]
\simeq\lp R/\p_\pi\rp[u],
\]
so we see
\begin{itemize}
\item
$\p_\pi B_1$ is prime and
\item
$[R]_{\p_\pi}=R/\p_\pi$ is algebraically closed in $\lc B_1\rc_{\p_\pi B_1}$.
\end{itemize}
Moreover, since $\gcd{(q,d)}=1$ by hypothesis, we know $[f]_{\p_\pi A}$ is separable.
And since $u$ is a primitive element for the extension $L_1/K$, \cref{prereduction} implies that $\lc B_1\rc_{\p_\pi B_1}/[A]_{\p_\pi A}$ is a Galois extension and $\rho_{\p_\pi B_1\vert\p_\pi A}$ is an isomorphism.
Thus, we set $N_1=q$, compute $c_{k,\phi,1,\lb0\rb}=q^m$, and apply \cref{effectivewedidit}:
for any monic irreducible polynomial $\pi\in R$ with $\deg{(\pi)}>\max{\lp\lb2md^2,4d\log_q{(d!)}\rb\rp}$,
\[
\frac{\lv\Per{\lp\P^1\lp[R]_{\p_\pi}\rp,[\phi]_{\p_\pi}\rp}\rv}
{\lv\P^1\lp[R]_{\p_\pi}\rp\rv}
=\frac{\lv\Per{\lp\P^1\lp[k]_{\vv_\pi}\rp[\phi]_{\vv_\pi}\rp}\rv}
{\lv\P^1\lp[k]_{\vv_\pi}\rp\rv}
<\fix_{n_{k,\phi,1,\lb0\rb,1}\lp\vv_\pi\rp}{\lp\rho\rp}
+\frac{7d}{\norm{(\vv_\pi)}^{\frac{1}{2}}}.
\]

Let $\pi_\alpha\in R$ be the minimal polynomial of $\alpha$ over $\F_q$, so that
\[
\frac{\lv\Per{\lp \P^1\lp\F_{q^r}\rp,X^d+\alpha^m\rp}\rv}
{\lv\P^1\lp\F_{q^r}\rp\rv}
=\frac{\lv\Per{\lp\P^1\lp\F_q(\alpha)\rp, X^d+\alpha^m\rp}\rv}
{\lv\P^1\lp\F_q(\alpha)\rp\rv}
=\frac{\lv\Per{\lp\P^1\lp[R]_{\p_\alpha}\rp,[\phi]_{\p_\alpha}\rp}\rv}
{\lv\P^1\lp[R]_{\p_\alpha}\rp\rv}.
\]
Since $\deg{\lp\pi_\alpha\rp}>\max{\lp\lb2md^2,4d\log_q{(d!)}\rb\rp}$ by hypothesis, the result follows by \ref{jfixed} and \cref{exdef}.
\end{proof}



Of course, for prime powers $q$ and positive integers $r$, most elements $\alpha\in\F_{q^r}$ have the property that $\F_q(\alpha)=\F_{q^r}$; we use this fact to deduce the following theorem on averages of periodic points.
For ease of notation, for the rest of the paper, we will write $\F_{q^r}^\text{prim}$ for the set $\lb\beta\in\F_{q^r}\mid\F_q(\beta)=\F_{q^r}\rb$.

\begin{theorem}\label{dandmave}
Suppose that $q$ is a prime power, that $r,m\in\Z_{\geq1}$, and that $d\in\Z_{\geq2}$.
If $q\equiv1\pmod{d}$ and $r>\max{\lp\lb2md^2,4d\log_q{(d!)}\rb\rp}$, then
\begin{align*}
\frac{1}{\lv\lp\F_{q^r}\rp^m\rv}\cdot\sum_{\beta\in\lp\F_{q^r}\rp^m}
{\frac{\lv\Per{\lp\P^1\lp\F_{q^r}\rp, X^d+\beta\rp}\rv}
{\lv\P^1\lp\F_{q^r}\rp\rv}}&\\
&\hspace{-175px}<\frac{4\log{(d)}\gcd{\lp q^r-1,m\rp}}
{(d-1)\lp\log{\lp\log{q^r}-\log{2}\rp}-\log{\max{\lp\lb\log{q^{2m}},\log{(d!)^4}\rb\rp}}\rp}
+\frac{(7d+2)\gcd{\lp q^r-1,m\rp}}{q^{\frac{r}{2}}}.
\end{align*}
In particular,
\[
\frac{1}{\lv\lp\F_{q^r}\rp^m\rv}\cdot
\sum_{\beta\in\lp\F_{q^r}\rp^m}
{\frac{\lv\Per{\lp\P^1\lp\F_{q^r}\rp, X^d+\beta\rp}\rv}
{\lv\P^1\lp\F_{q^r}\rp\rv}}
=O_{q,m,d}\lp\frac{1}{\log{\log{q^r}}}\rp.
\]
\end{theorem}
\begin{proof}
We begin by setting
\[
\kappa_{q,d,m}(r)
=\frac{4\log{d}}
{(d-1)\lp\log{\lp\log{q^r}-\log{2}\rp}
-\log{\max{\lp\lb\log{q^{2m}},\log{(d!)^4}\rb\rp}}\rp}
+\frac{7d}{q^{\frac{r}{2}}}.
\]
Write $\mathcal{P}_q(r)$ for the set of monic irreducible polynomials in $\F_q[s]$ of degree $r$ and for any such polynomial $\pi$, write $S_q(\pi)$ for its roots in $\F_{q^r}$.
If we let $\phi(X)=X^d+s^m\in\F_q[s][X]$, then we see that
\begin{align*}
\sum_{\beta\in\lp\F_{q^r}^\text{prim}\rp^m}
{\frac{\lv\Per{\lp\P^1\lp\F_{q^r}\rp, X^d+\beta\rp}\rv}
{\lv\P^1\lp\F_{q^r}\rp\rv}}
&\leq\sum_{\pi\in\mathcal{P}_q(r)}
{\sum_{\alpha\in S_q(\pi)}
{\frac{\lv\Per{\lp\P^1\lp\F_{q^r}\rp, X^2+\alpha^m\rp}\rv}
{\lv\P^1\lp\F_{q^r}\rp\rv}}}&&\\
&=\sum_{\pi\in\mathcal{P}_q(r)}
{r\cdot\frac{\lv\Per{\lp\P^1\lp\F_q[s]/\pi\F_q[s]\rp,\lc\phi\rc_\pi\rp}\rv}
{\lv\P^1\lp\F_{q^r}\rp\rv}}&&\\
&<r\lv\mathcal{P}_q(r)\rv
\kappa_{q,d,m}(r)
&&\hspace{-40px}\text{(by \cref{dandm})}\\
&\leq q^r\kappa_{q,d,m}(r)
&&\hspace{-40px}\text{(since $\lv\mathcal{P}_q(r)\rv\leq r^{-1}q^r$)}.
\end{align*}

Next, we recall that there are at most $2q^{r/2}$ elements of $\F_{q^r}$ whose minimal polynomial is of degree less than $r$.
Thus,
\[
\sum_{\beta\in\lp\F_{q^r}\setminus\F_{q^r}^\text{prim}\rp^m}
{\frac{\lv\Per{\lp\P^1\lp\F_{q^r}\rp, X^2+\beta\rp}\rv}
{\lv\P^1\lp\F_{q^r}\rp\rv}}
\leq\sum_{\beta\in\lp\F_{q^r}\setminus\F_{q^r}^\text{prim}\rp^m}
{\hspace{-18px}1\hspace{18px}}
\leq2q^{\frac{r}{2}},
\]
so that
\[
\frac{1}{1+\frac{q^r}{\gcd{\lp q^r-1,m\rp}}}
\lp q^r\kappa_{q,d,m}(r)
+2q^{\frac{r}{2}}\rp
<\gcd{\lp q^r-1,m\rp}
\lp\kappa_{q,d,m}(r)
+\frac{2}{q^{\frac{r}{2}}}\rp,
\]
as desired.
\end{proof}

Finally, we address the specific case of quadratic polynomials.

\begin{proof}[Proof of \cref{degtwoave}]
Let $Q=\lb f\in\F_{q^r}[X]\mid\deg{f}=2\rb$ and $U=\lb X^2+\delta\mid\delta\in\F_{q^r}\rb$.
Since $\Char{\lp\F_q\rp}\neq2$, for any $\alpha\in\F_{q^r}\setminus\lb0\rb$ and $\beta\in\F_{q^r}$ we may define the following coordinate change on $\P^1\lp\F_{q^r}\rp$:
\[
\mu_{\alpha,\beta}:x\mapsto\alpha x+\frac{\beta}{2}.
\]
Next, we set
\begin{align*}
\mu\colon\hspace{70px} Q&\to U\\
\alpha X^2+\beta X+\gamma&\mapsto X^2-\frac{\beta^2-4\alpha\gamma-2\beta}{4}.
\end{align*}
This map is clearly surjective.
Moreover, if $\delta\in\F_{q^r}$, then
\[
\lv\mu^{-1}\lp X^2+\delta\rp\rv=
\begin{cases}
q^{2r}-q^r&\text{if }1-4\delta\text{ is not a square in }\F_{q^r}\\
q^{2r}+q^r&\text{if }1-4\delta\text{ is a nonzero square in }\F_{q^r}\\
q^{2r}&\text{if }1-4\delta=0.
\end{cases}
\]
Note that for any $f\in Q$, if $f(X)=\alpha X^2+\beta X+\gamma$, then
\[
\mu(f)=\mu_{\alpha,\beta}\circ f\circ\mu_{\alpha,\beta}^{-1},
\]
so that
\[
\lv\Per{\lp\P^1\lp\F_{q^r}\rp, f\rp}\rv
=\lv\Per{\lp\P^1\lp\F_{q^r}\rp,\mu(f)\rp}\rv.
\]
Setting
\[
\kappa_q(r)
=\frac{\log{16}}{\log{\lp\log{q^r}-\log{2}\rp}-\log{\max{\lp\lb\log{q^2},\log{16}\rb\rp}}}+\frac{14}{q^{\frac{r}{2}}},
\]
we apply \cref{dandm} to see that
\begin{align*}
&\frac{1}{\lv Q\rv}
\sum_{f\in Q}
{\frac{\lv\Per{\lp\P^1\lp\F_{q^r}\rp, f\rp}\rv}{\lv\P^1\lp\F_{q^r}\rp\rv}}\\
&\hspace{50px}<
\frac{q^{2r}+q^r}{\lv Q\rv}\lp
\sum_{\beta\in\F_{q^r}^\text{prim}}
{\frac{\lv\Per{\lp\P^1\lp\F_{q^r}\rp, X^2+\beta\rp}\rv}
{\lv\P^1{\lp\F_{q^r}\rp}\rv}}
+\sum_{\beta\in\F_{q^r}\setminus\F_{q^r}^\text{prim}}
{\frac{\lv\Per{\lp\P^1\lp\F_{q^r}\rp, X^2+\beta\rp}\rv}
{\lv\P^1{\lp\F_{q^r}\rp}\rv}}\rp\\
&\hspace{50px}<\frac{q^{2r}+q^r}{\lv Q\rv}\lp
q^r\kappa_q(r)+2q^{\frac{r}{2}}\rp\\
&\hspace{50px}=\frac{q^r+1}{q^r-1}\lp\kappa_q(r)+\frac{2}{q^{\frac{r}{2}}}\rp.
\end{align*}
\end{proof}

\noindent We now need only apply elementary estimates to \cref{degtwoave} to prove \cref{easytostate}.

\begin{proof}[Proof of \cref{easytostate}]
Suppose that $p\geq5$ and set $x=\log{p^r}-\log{2}$.
Since $r>6\log{p}$, we see that
\[
x>5\lp\log{p}\rp^2
>\frac{\lp\log{p^2}\rp^2+\sqrt{\lp\log{p^2}\rp^4+4\lp\log{p^2}\rp^2\log{2}}}{2},
\]
so that
\[
x^2-\lp\log{p^2}\rp^2x-\lp\log{p^2}\rp^2\log{2}>0.
\]
This implies that
\[
\frac{x}{\log{p^2}}
>\lp x+\log{2}\rp^{\frac{1}{2}}
>\lp x+\log{2}\rp^{\frac{\log{16}}{6}}
\]
so
\[
\frac{6}{\log{\lp x+\log{2}\rp}}
>\frac{\log{16}}{\log{x}-\log{\log{p^2}}}.
\]
Since
\[
\frac{5}{\log{\log{p^r}}}
>\frac{16}{p^{\frac{r}{2}}},
\]
the result follows from \cref{degtwoave}.
When $p=3$, the proof is similar.
\end{proof}

\section*{Acknowledgements} \label{Acknowledgements}

I would like to thank Andrew Bridy and Jamie Juul for useful correspondence.
I also owe Tom Tucker my gratitude for taking the time to discuss the finer points of~\cite{JKMT}.
Finally, I thank the anonymous reviewer for many helpful comments.

\bibliography{FewPeriodicPoints}
\bibliographystyle{amsalpha}

\end{document}